\documentclass[12pt]{amsart}

\usepackage[matrix,arrow,curve,frame]{xy}
\usepackage{amsmath,amsthm,amssymb,enumerate}
\usepackage{latexsym}
\usepackage{amscd}
\usepackage[colorlinks=false]{hyperref}
\usepackage[mathscr]{eucal}

\newtheorem{thm}[subsection]{Theorem}

\newtheorem{cor}[subsection]{Corollary}
\newtheorem{lemma}[subsection]{Lemma}

\newtheorem{remark}[subsection]{Remark}

\theoremstyle{definition}

\numberwithin{equation}{section}

\def\phi{{\varphi}}

\def\ra{\rightarrow}

\def\cA{{\mathcal A}}

\def\cH{{\mathcal H}}

\def\cN{{\mathcal N}}
\def\cO{{\mathcal O}}

\def\cS{{\mathcal S}}

\def\cU{{\mathcal U}}
\def\cV{{\mathcal V}}
\def\cW{{\mathcal W}}

\def\gl{{\mathfrak l}}

\def\gp{{\mathfrak p}}

\def\gs{{\mathfrak s}}

\newfont{\german}{eufm10}

\begin{document}
\pagestyle{plain}

\title
{Singular support of a vertex algebra and the arc space of its associated scheme}

\author{Tomoyuki Arakawa}
\address{Research Institute for Mathematical Sciences, Kyoto University}
\email{arakawa@kurims.kyoto-u.ac.jp}
\thanks{T. A. is supported by JSPS KAKENHI Grants
\#17H01086  and \#17K18724}

\author{Andrew R. Linshaw} 
\address{University of Denver}
\email{andrew.linshaw@du.edu}
\thanks{A. L. is supported by Simons Foundation Grant \#318755}

\thanks{We thank Julien Sebag for helpful comments on an earlier draft of this paper.}

\maketitle

\begin{center}
{\small {\em Dedicated to Professor Anthony Joseph on his seventy-fifth birthday}}
\end{center}

\begin{abstract} Attached to a vertex algebra $\cV$ are two geometric objects. The {\it associated scheme} of $\cV$ is the spectrum of Zhu's Poisson algebra $R_{\cV}$. The {\it singular support} of $\cV$ is the spectrum of the associated graded algebra $\text{gr}(\cV)$ with respect to Li's canonical decreasing filtration. There is a closed embedding  from the singular support to the arc space of the associated scheme, which is an isomorphism in many interesting cases. In this note we give an example of a non-quasi-lisse vertex algebra whose associated scheme is reduced, for which the isomorphism is not true as schemes but true as varieties.
\end{abstract}


\section{Introduction}
Attached to a vertex algebra $\cV$ are two geometric objects. The {\it associated scheme} 
$\tilde{X}_{\cV}$ of $\cV$ is the spectrum of commutative algebra $R_{\cV}$,
which is 
an
affine Poisson scheme of finite type\footnote{provided that $\cV$ is finitely strongly generated}.
The {\it singular support} 
$\text{SS}(\cV)$
of $\cV$ is the spectrum of the associated graded algebra $\text{gr}(\cV)$ with respect to Li's canonical decreasing filtration,
which is 
a {\em vertex Poisson} scheme of infinite type\footnote{unless $\cV$ is finite-dimensional}.
There is a closed embedding 
$$\Phi:\text{SS}(\cV)\hookrightarrow (\tilde{X}_{\cV})_\infty$$
from the singular support to the arc space $ \tilde{X}_\infty$ of the associated scheme, 
which is an isomorphism in many interesting cases.

Originally
Zhu \cite{Zh} introduced the algebra $R_{\cV}$ to define a certain finiteness condition on a vertex algebra.
Recall that a vertex algebra $\cV$ is called {\em lisse} (or $C_2$-cofinite) if 
$\dim \tilde{X}_{\cV}=0$.
Using the map $\Phi$ one can show that this condition is equivalent to that 
$\dim \text{SS}(\cV)=0$,
and hence,
the lisse condition is a natural finiteness condition (\cite{ArI}).
It is known that lisse vertex (operator) algebras have many nice properties,
such as modular invariance property of characters (\cite{Zh,Mi}),
and this condition has been assumed in many significant theories of vertex (operator) algebras.
However, 
recently 
non-lisse 
vertex algebras 
have caught a lot of  attention  
due to the {\em Higgs branch conjecture} by Beem and Rastelli \cite{BR},
which states that 
 the reduced scheme $X_{\cV}$ of $\tilde{X}_{\cV}$
 should be isomorphic to the Higgs branch of a four-dimensional $N=2$ superconformal field theory
 $\mathcal{T}$ if $\cV$ obtained from $\mathcal{T}$ by the correspondence discovered by \cite{BLL+},
 see the survey articles \cite{ArII, ArIII} and the references therein.

It is natural to ask whether the map $\Phi$ is always an isomorphism,
and if not, whether $\Phi$ defines an isomorphism as varieties.
Very recently counterexamples to the first question were found by 
van Ekeren and  Heluani \cite{EH}
in the case that $\cV$ is lisse in their study of chiral homology of elliptic curves. It was also shown recently in \cite{AMII} that the map $\Phi$ defines an isomorphism as varieties 
if $\cV$ is {\em quasi-lisse}, that is, the Poisson variety $X_{\cV}$ has finitely many symplectic leaves. In this note we give an example
of a non-quasi-lisse
vertex algebra whose associated scheme is reduced, for which $\Phi$ is not an isomorphism of schemes, but still defines an isomorphism of varieties. We remark that by tensoring one of the lisse examples in \cite{EH} with any non-quasi-lisse vertex algebra, one can trivially obtain a non-quasi-lisse example. However, all such examples have the property that the associated scheme is nonreduced.
\section{Vertex algebras} \label{sect:va}
We assume that the reader is familiar with vertex algebras, which have been discussed from various points of view in the literature \cite{B,FLM,K,FBZ}. Given an element $a$ in a vertex algebra $\cV$, the field associated to $a$ via the state-field correspondence is denoted by $$a(z) = \sum_{n\in \mathbb{Z}} a(n) z^{-n-1} \in \text{End}(\cV)[[z,z^{-1}]].$$ Throughout this paper, we shall identify $\cV$ with the corresponding space of fields. Given $a,b \in \cV$, the {\it operators product expansion} (OPE) formula is given by
$$a(z)b(w)\sim\sum_{n\geq 0}(a_{(n)} b)(w)\ (z-w)^{-n-1}.$$ Here $(a_{(n)} b)(w) = \text{Res}_z [a(z), b(w)](z-w)^n$ where $$[a(z), b(w)] = a(z) b(w)\ -\ (-1)^{|a||b|} b(w)a(z),$$ and $\sim$ means equal modulo terms which are regular at $z=w$. The {\it normally ordered product} $:a(z)b(z):$ is defined to be $$a(z)_-b(z)\ +\ (-1)^{|a||b|} b(z)a(z)_+,$$ where $$a(z)_-=\sum_{n<0}a(n)z^{-n-1},\qquad a(z)_+=\sum_{n\geq
0}a(n)z^{-n-1}.$$ We usually omit the formal variable $z$ and write $:a(z)b(z): \ = \ :ab:$, when no confusion can arise. For $a_1,\dots, a_k \in \cV$, the iterated normally ordered product is defined inductively by
\begin{equation} \label{iteratedwick} :a_1 a_2 \cdots a_k:\ =\ :a_1 \big(  :a_2\cdots a_k: \big).\end{equation}
A subset $S=\{a_i|\ i\in I\}$ of $\cV$  is said to {\it strongly generate} $\cV$, if $\cV$ is spanned by the set of normally ordered monomials $$:\partial^{k_1} a_{i_1}\cdots \partial^{k_m} a_{i_m}:,\qquad i_1,\dots,i_m \in I, \qquad  k_1,\dots,k_m \geq 0.$$ If $S$ is an ordered strong generating set $\{\alpha^1, \alpha^2,\dots\}$, we say that $S$ {\it freely generates} $\cV$, if $\cV$ has a PBW basis consisting of 

\begin{equation} \label{freegen} \begin{split} & :\partial^{k^1_1} \alpha^{i_1} \cdots \partial^{k^1_{r_1}}\alpha^{i_1} \partial^{k^2_1} \alpha^{i_2} \cdots \partial^{k^2_{r_2}}\alpha^{i_2}
 \cdots \partial^{k^n_1} \alpha^{i_n} \cdots \partial^{k^n_{r_n}} \alpha^{i_n}:,\qquad 
 1\leq i_1 < \dots < i_n,
 \\ & k^1_1\geq k^1_2\geq \cdots \geq k^1_{r_1},\qquad k^2_1\geq k^2_2\geq \cdots \geq k^2_{r_2},  \ \ \cdots,\ \  k^n_1\geq k^n_2\geq \cdots \geq k^n_{r_n},
 \\ &  k^{t}_1 > k^t_2 > \dots > k^t_{r_t}\ \  \text{whenever} \ \ \alpha^{i_t}\ \ \text{is odd}. 
 \end{split} \end{equation} 
In particular, the monomials \eqref{freegen} are linearly independent, so there are no nontrivial normally ordered polynomial relations among the generators and their derivatives.

\subsection{$\beta\gamma$-system}
The $\beta\gamma$-system $\cS$ is freely generated by even fields $\beta,\gamma$ satisfying
\begin{equation}
\begin{split}
\beta (z)\gamma (w) &\sim (z-w)^{-1},\qquad \gamma (z)\beta (w)\sim - (z-w)^{-1},\\ 
\beta (z)\beta (w) &\sim 0,\qquad\qquad\qquad \gamma (z)\gamma (w)\sim 0.\end{split} \end{equation}
It has Virasoro element $L^{\cS} = \frac{1}{2} \big(:\beta \partial\gamma : - :\partial\beta \gamma :\big)$ of central charge $c=-1$, under which $\beta$, $\gamma$ are primary of weight $\frac{1}{2}$.

\subsection{$\cW_3$-algebra}
The $\cW_3$-algebra $\cW^c_3$ with central charge $c$ was introduced by Zamolodchikov \cite{Za}. It is an extension of the Virasoro algebra, and is freely generated by a Virasoro field $L$ and an even weight $3$ primary field $W$. In fact, $\cW^c_3$ is isomorphic to the principal $\cW$-algebra $\cW^k(\gs\gl_3, f_{\text{prin}})$ where $c = 2 - \frac{24(k+2)^2}{k+3}$. For generic values of $c$, $\cW_3^c$ is simple, but for certain special values it has a nontrivial ideal. In this paper, we only need the case $c=-2$, which is nongeneric. We shall denote the simple graded quotient of $\cW^{-2}_3$ by $\cW$ for the rest of the paper. Since $\cW_3^{-2}$ has a nontrivial ideal, $\cW$ is strongly but not freely generated by $L,W$. 

There is a useful embedding $i:\cW \ra \cS$ due to Wang \cite{WaI}, given by
\begin{equation} \label{eq:embeddingW} \begin{split} & L \mapsto  \frac{1}{2} : \beta\beta\gamma\gamma: + :\beta (\partial \gamma): - :(\partial \beta) \gamma:, \\ & W \mapsto \frac{1}{4 \sqrt{2}} \bigg( 2 :\beta^3 \gamma^3: + 9 : \beta^2 (\partial \gamma) \gamma: + 3 :\beta \partial^2 \gamma: - 9:(\partial \beta) \beta \gamma^2: - 12 \partial \beta)(\partial \gamma): +3 :(\partial^2 \beta) \gamma: \bigg),\end{split}\end{equation} and we shall identify $\cW$ with its image in $\cS$. In fact, $\cW$ is precisely the subalgebra of $\cS$ that commutes with the Heisenberg algebra generated by $:\beta\gamma:$. Note that $W$ is normalized so that it satisfies 
$$W(z) W(w) \sim -\frac{9}{8} (z-w)^{-6} + \frac{27}{8} L(w) (z-w)^{-4} +  \frac{27}{16} \partial L(w) (z-w)^{-3} $$ $$ +\bigg(\frac{9}{2} :LL: - \frac{27}{32} \partial^2 L  \bigg)(w)(z-w)^{-2} +\bigg(\frac{9}{2} :(\partial L)L: - \frac{3}{16} \partial^3 L \bigg)(w) (z-w)^{-1}.$$
This normalization is nonstandard but convenient for our purposes.

\subsection{Zhu's commutative algebra and the associated variety}
Given a vertex algebra $\cV$, define \begin{equation} \label{def:zhucomm} C(\cV) = \text{Span}\{a_{(-2)} b|\ a,b\in \cV \},\qquad R_{\cV} = \cV / C(\cV).\end{equation} It is well known that $R_{\cV}$ is a commutative, associative algebra with product induced by the normally ordered product \cite{Zh}. Also, if $\cV$ is graded by conformal weight, $R_{\cV}$ inherits this grading. Define the associated scheme \begin{equation} \tilde{X}_{\cV} = \text{Spec}(R_{\cV}),\end{equation} and the associated variety \begin{equation} X_{\cV} = \text{Specm}(R_{\cV}) = (\tilde{X}_{\cV})_{\text{red}}.\end{equation} Here $(\tilde{X}_{\cV})_{\text{red}}$ denotes the reduced scheme of $\tilde{X}_{\cV}$. If $\{\alpha_i|\ i\in I\}$ is a strong generating set for $\cV$, the images of these fields in $R_{\cV}$ will generate $R_{\cV}$ as a ring. In particular, $R_{\cV}$ is finitely generated if and only if $\cV$ is strongly finitely generated. 

Since the $\beta\gamma$-system $\cS$ is freely generated by $\beta,\gamma$, $R_{\cS} \cong \mathbb{C}[b,g]$, where $b,g$ denote the images of $\beta,\gamma$ in $R_{\cS}$. On the other hand, since $\cW$ is not freely generated by $L,W$, the structure of $R_{\cW}$ is more complicated. 

\begin{lemma} \label{lem:rw} Let $\ell,w$ denote the images of $L,W$ in $R_{\cW}$. Then $R_{\cW} \cong \mathbb{C}[\ell,w] / \langle w^2 - \ell^3\rangle$. \end{lemma}
\begin{proof} Since $\cW$ is strongly generated by $L,W$, $R_{\cW}$ is generated by $\ell,w$, so $R_{\cW} \cong  \mathbb{C}[\ell,w] / I$ for some ideal $I$. By Lemma 2.1 of \cite{WaII}, we have the following normally ordered relation in $\cW$ at weight $6$:
\begin{equation} \label{eq:we6relationW} :W^2: - :L^3: - \frac{7}{8}:( \partial^2 L)L: - \frac{19}{32} :(\partial L)^2: \ = 0.\end{equation} Note that \eqref{eq:we6relationW} differs slightly from the formula in \cite{WaII} because our normalization of $W$ is different. It follows that $w^2 - \ell^3  \in I$. 

To see that $I \subseteq \langle w^2 - \ell^3 \rangle$, let $p = p(\ell,w) \in I$. Without loss of generality, we may assume $p$ is homogeneous of weight $d$. It must come from a normally ordered polynomial relation $$P = P(L,\partial L, \dots, W, \partial W, \dots) = 0$$ of weight $d$ in $\cW$ among $L,W$ and their derivatives. The monomials of $p$ correspond to the normally ordered monomials of $P$ which do not lie in $C(\cW)$, and have the form 
\begin{equation} \label{liwj} :L^i W^j:,\qquad 2i + 3j = d.\end{equation} Using \eqref{eq:we6relationW} repeatedly, we can rewrite this relation in the form $$P' = P'(L,\partial L, \dots, W, \partial W, \dots) = 0,$$  where all terms of the form \eqref{liwj} that appear either have $j = 0$ or $j=1$. In fact, since $P'$ is homogeneous of weight $d$, we must have $j = 0$ if $d$ is even, and $j = 1$ if $d$ is odd, so only one such term can appear. If this term appears with nonzero coefficient, as a normally ordered polynomial in $\beta, \gamma$ and their derivatives, it will contribute the term $:\beta^{2i+3j} \gamma^{2i + 3j}:$, which cannot be canceled. This contradicts $P' = 0$, so each monomial in $P'$ must lie in $C(\cW)$. Equivalently, $p \in \langle w^2 - \ell^3\rangle$.
 \end{proof}

\section{Jet schemes and arc spaces} \label{sect:jet} 
We recall some basic facts about jet schemes, following the notation in \cite{EM}. Let $X$ be an irreducible scheme over $\mathbb{C}$ of finite type. The first jet scheme $X_1$ is the total tangent space of $X$, and for $m>1$ the jet schemes $X_m$ are higher-order generalizations which are determined by their functor of points. Given a $\mathbb{C}$-algebra $A$, we have a bijection
$$\text{Hom} (\text{Spec}  (A), X_m) \cong \text{Hom} (\text{Spec}  (A[t]/\langle t^{m+1}\rangle ), X).$$ Thus the $\mathbb{C}$-valued points of $X_m$ correspond to the $\mathbb{C}[t]/\langle t^{m+1}\rangle$-valued points of $X$. For $p>m$, we have projections $\pi_{p,m}: X_p \rightarrow X_m$ and $\pi_{p,m} \circ \pi_{q,p} = \pi_{q,m}$  when $q>p>m$. The assignment $X\mapsto X_m$ is functorial, and a morphism $f:X\ra Y$ induces $f_m: X_m \ra Y_m$ for all $m\geq 1$. If $X$ is nonsingular, $X_m$ is irreducible and nonsingular for all $m$. If $X,Y$ are nonsingular and $f:X\ra Y$ is a smooth surjection, $f_m$ is surjective for all $m$. 

For an affine scheme $X=\text{Spec}(R)$ where $R= \mathbb{C}[y_1,\dots,y_r] / \langle f_1,\dots, f_k\rangle$, $X_m$ is also affine and we can give explicit equations for $X_m$ as follows. Define variables $y_1^{(i)},\dots y_r^{(i)}$ for $i=0,\dots, m$, and define a derivation $D$ by 
\begin{equation}\label{actiond} D(y_j^{(i)}) = \bigg\{ \begin{matrix} y_j^{(i+1)} & 0\leq i < m \cr 0 & i = m \end{matrix},\end{equation} 
which specifies its action on all of $\mathbb{C}[y_1^{(i)},\dots, y_r^{(i)}]$, for $0\leq i \leq m$. In particular, $f_{\ell}^{(i)} = D^i ( f_{\ell})$ is a well-defined polynomial in $\mathbb{C}[y_1^{(i)},\dots, y_r^{(i)}]$. Letting 
$$R_m = \mathbb{C}[y_1^{(i)},\dots, y_r^{(i)}]/ \langle f_1^{(i)},\dots, f_k^{(i)} \rangle,$$ 
we have $X_m\cong \text{Spec}(R_m)$. By identifying $y_j$ with $y_j^{(0)}$, we may identify $R$ with a subalgebra of $R_m$. There is a $\mathbb{Z}_{\geq 0}$-grading on $R_m$ which we call {\it height}, given by
\begin{equation} \label{height} R_m = \bigoplus_{n\geq 0} R_m[n],\qquad \text{ht}(y^{(i)}_j) = i.\end{equation} For all $m$, $R_m[0] = R$ and $R_m[n]$ is an $R$-module.

Given a scheme $X$, define \begin{equation} \label{defarcspace} X_{\infty}  = \lim_{\leftarrow } X_m,\end{equation} which is known as the {\it arc space} of $X$. For a $\mathbb{C}$-algebra $A$, we have a bijection
$$\text{Hom} (\text{Spec} (A), X_{\infty} ) \cong \text{Hom} (\text{Spec} (A[[t]]), X),$$ so the $\mathbb{C}$-valued points of $X_{\infty}$ correspond to the $\mathbb{C}[[t]]$-valued points of $X$. If $X = \text{Spec}(R)$ as above, $$X_{\infty} \cong \text{Spec} (R_{\infty}), \text{where} \ R_{\infty}  = \mathbb{C}[y_1^{(i)},\dots, y_r^{(i)}] / \langle f_1^{(i)}, \dots, f_k^{(i)} \rangle.$$ Here $i\geq 0$, and $D (y^{(i)}_j) = y^{(i+1)}_j$ for all $i$. 

By a theorem of Kolchin \cite{Kol}, $X_{\infty}$ is irreducible if $X$ is irreducible. However, even if $X$ is irreducible and reduced, $X_{\infty}$ need not be reduced. The following result is due to Sebag (see Example 8 of \cite{SI}, as well as more general results in \cite{SII}), but we include a proof for the benefit of the reader.

\begin{lemma} For $X = \text{Spec} (\mathbb{C}[\ell, w ] / \langle w^2 - \ell^3\rangle ) = \tilde{X}_{\cW}$, $X_{\infty}$ is not reduced.
\end{lemma}

\begin{proof} We have \begin{equation} \label{eq:rinfty} X_{\infty} \cong \text{Spec}(R_{\infty}),\qquad R_{\infty} = \mathbb{C}[\ell^{(0)}, \ell^{(1)},\dots, w^{(0)}, w^{(1)},\dots] / \langle f^{(0)}, f^{(1)},\dots \rangle,\end{equation} where $f^{(0)} = (\ell^{(0)})^3 - (w^{(0)})^2$. Consider the element \begin{equation} \label{def:r1} r_1 = 3 \ell^{(1)} w^{(0)} - 2 \ell^{(0)} w^{(1)} \in \mathbb{C}[\ell^{(0)}, \ell^{(1)},\dots, w^{(0)}, w^{(1)},\dots].\end{equation} First, $r_1 \notin \langle f^{(0)}, f^{(1)},\dots \rangle$ since no element of this ideal has leading term of degree $2$. However, $(r_1)^3 \in \langle f^{(0)}, f^{(1)},\dots \rangle$; a calculation shows that
\begin{equation} \begin{split} (r_1)^3 & = \bigg(-81 \ell^{(0)} \ell^{(1)} \ell^{(2)} w^{(0)} - \frac{27}{2} (\ell^{(0)})^2 \ell^{(3)} w^{(0)} + 18 (\ell^{(0)})^2 \ell^{(2)} w^{(1)} - 4 (w^{(1)})^3 + 15 w^{(0)} w^{(1)} w^{(2)} 
\\ & + 9 (w^{(0)})^2 w^{(3)}\bigg) f^{(0)} + \bigg( \frac{9}{2} (\ell^{(0)})^2 \ell^{(2)} w^{(0)} + 12 (\ell^{(0)})^2 \ell^{(1)} w^{(1)} - 7 w^{(0)} (w^{(1)})^2 - 3 (w^{(0)})^2 w^{(2)}\bigg)f^{(1)} \\ & + \bigg( -\frac{9}{2}(\ell^{(0)})^2 \ell^{(1)} w^{(0)} - 6 (\ell^{(0)})^3 w^{(1)} + 9 (w^{(0)})^2 w^{(1)} \bigg) f^{(2)} +\bigg( \frac{9}{2} (\ell^{(0)})^3 w^{(0)} - \frac{9}{2} (w^{(0)})^3 \bigg) f^{(3)}.\end{split} \end{equation} Therefore regarded as an element of $R_{\infty}$, $r_1 \neq 0$ but $(r_1)^3 = 0$. \end{proof}

It is well known that in characteristic zero, for any affine scheme $X$, the nilradical $\cN \subseteq \cO(X_{\infty})$ is a differential ideal; in other words, $D(\cN) \subseteq \cN$. A natural question (see \cite{KS}) is whether $\cN$ is finitely generated as a differential ideal, and whether an explicit generating set can be found. In general, $\cN$ need not be finitely generated; this was shown for $X = \text{Spec}(\mathbb{C}[x,y] / \langle xy\rangle )$ in \cite{BS}. In the example $X =\text{Spec} (\mathbb{C}[\ell, w ] /\langle w^2 - \ell^3\rangle )$, a calculation shows that in addition to $r_1$, \begin{equation} \label{def:r2}  r_2 = (w^{(1)})^2  - \frac{9}{4} \ell^{(0)} (\ell^{(1)})^2 \end{equation} does not lie in $\langle f^{(0)}, f^{(1)},\dots \rangle$, but $(r_2)^3$ does. So $r_2$ is another nontrivial element of $\cN$. We expect that $\cN$ is generated as a differential ideal by $r_1$ and $r_2$.

The following characterization of $\cN$ in this example will also be useful to us.

\begin{lemma} Let $X = \text{Spec} (\mathbb{C}[\ell, w ] / \langle w^2 - \ell^3\rangle )$ and let $t$ be a coordinate function on $\mathbb{C}$. Consider the map
\begin{equation} \label{eq:birational} \mathbb{C} \ra X,\qquad t \mapsto (t^2, t^3), \end{equation} and the induced homomorphism
\begin{equation} \label{eq:induced} \phi: \cO(X_{\infty}) \ra \cO(\mathbb{C}_{\infty}),\qquad \ell^{(0)} \mapsto (t^{(0)})^2,\qquad w^{(0)} \mapsto (t^{(0)})^3. \end{equation} Then $\cN = \text{ker}(\phi)$.
\end{lemma} 

\begin{proof} Since \eqref{eq:birational} is birational, the map $\mathbb{C}_{\infty} \ra X_{\infty}$ on arc spaces induced by \eqref{eq:birational} is dominant,
see Proposition 3.2 of \cite{EM}.
 Therefore $\text{ker}(\phi) \subseteq \cN$. On the other hand, $\cN \subseteq \text{ker}(\phi)$ since $\cO(\mathbb{C}_{\infty}) \cong \mathbb{C}[t^{(0)}, t^{(1)},\dots]$, which is an integral domain. \end{proof}

\section{Li's filtration and singular support}
For any vertex algebra $\cV$, we have Li's canonical decreasing filtration
$$F^0(\cV) \supseteq F^1(\cV) \supseteq \cdots,$$ where $F^p(\cV)$ is spanned by elements of the form
$$:\partial^{n_1} a^1 \partial^{n_2} a^2 \cdots \partial^{n_r} a^r:,$$ 
where $a^1,\dots, a^r \in \cV$, $n_i \geq 0$, and $n_1 + \cdots + n_r \geq p$ \cite{LiI}. Clearly $\cV = F^0(\cV)$ and $\partial F^i(\cV) \subseteq F^{i+1}(\cV)$. Set $$\text{gr}(\cV) = \bigoplus_{p\geq 0} F^p(\cV) / F^{p+1}(\cV),$$ and for $p\geq 0$ let 
$$\sigma_p: F^p(\cV) \ra F^p(\cV) / F^{p+1}(\cV) \subseteq \text{gr}(\cV)$$ be the projection. Note that $\text{gr}(\cV)$ is a graded commutative algebra with product
$$\sigma_p(a) \sigma_q(b) = \sigma_{p+q}(a_{(-1)} b),$$ for $a \in F^p(\cV)$ and $b \in F^q(\cV)$. We say that the subspace $F^p(\cV) / F^{p+1}(\cV)$ has {\it height} $p$. Note that $\text{gr}(\cV)$ has a differential $\partial$ defined by $$\partial( \sigma_p(a) ) = \sigma_{p+1} (\partial a),$$ for $a \in F^p(\cV)$. Finally, $\text{gr}(\cV)$ has the structure of a Poisson vertex algebra \cite{LiI}; for $n\geq 0$, we define
$$\sigma_p(a)_{(n)} \sigma_q(b) = \sigma_{p+q-n} a_{(n)} b.$$
Zhu's commutative algebra $R_{\cV}$ is isomorphic to the subalgebra $F^0(\cV) / F^1(\cV)\subseteq \text{gr}(\cV)$, since $F^1(\cV)$ coincides with the space $C(\cV)$ defined by \eqref{def:zhucomm}. Moreover, $\text{gr}(\cV)$ is generated by $R_{\cV}$ as a differential graded commutative algebra \cite{LiI}. Since $\tilde{X}_{\cV} = \text{Spec}(R_{\cV})$, there is always a surjective homomorphism of differential graded rings
\begin{equation} \label{dgalgebras} \Phi_{\cV}: \cO((\tilde{X}_{\cV})_{\infty}) \ra \text{gr}(\cV),\end{equation} where the grading on $ \cO((\tilde{X}_{\cV})_{\infty})$ is given by \eqref{height}. Define the {\it singular support} \begin{equation} \label{def:ss} \text{SS}(\cV)=\text{Spec}(\text{gr}(\cV)),\end{equation} which is then a subscheme of $(\tilde{X}_{\cV})_{\infty}$. A natural question which was raised by Arakawa and Moreau \cite{AMI} is whether the map \eqref{dgalgebras} is always an isomorphism. This is true in many examples
and it was recently shown in \cite{AMII} to hold as varieties when $\cV$ is quasi-lisse,
that is,
if $X_{\cV}$ has finitely many symplectic leaves,
see \cite{AK} for the details.
We note that the vertex algebra $\cW$  is not quasi-lisse.

\section{Main result}

\begin{thm} For the vertex algebra $\cW$, the map $\Phi_{\cW}: \cO((\tilde{X}_{\cW})_{\infty}) \ra \text{gr}(\cW)$ is not injective, so $(\tilde{X}_{\cW})_{\infty}$ and $\text{SS}(\cW)$ are not isomorphic as schemes.
\end{thm}

\begin{proof} As before, we use the notation $$\cO((\tilde{X}_{\cW})_{\infty}) \cong R_{\infty} = \mathbb{C}[\ell^{(0)}, \ell^{(1)},\dots, w^{(0)}, w^{(1)},\dots] / \langle f^{(0)}, f^{(1)},\dots \rangle.$$ We use the same notation $\partial^i L,\partial^i W$ to denote the images of the fields $\partial^iL,\partial^iW \in \cW$ in the subspace $F^i(\cW) / F^{i+1}(\cW)$ of $\text{gr}(\cW)$. We therefore may identify $\text{gr}(\cW)$ with a quotient of the polynomial ring $\mathbb{C}[L, \partial L,\dots, W, \partial W, \dots]$. In this notation, $\Phi_{\cW}(\ell^{(0)}) = L$ and $\Phi_{\cW}(w^{(0)}) =W$.

 We will show that the nilpotent elements $r_1$ and $r_2$ in $\cO((\tilde{X}_{\cW})_{\infty})$ given by \eqref{def:r1} and \eqref{def:r2}, lie in $\text{ker}(\Phi_{\cW})$. By Lemma 2.1 of \cite{WaII}, we have the following relation in $\cW$ at weight $6$:
$$3 :(\partial L) W: - 2 :L (\partial W):  +\frac{1}{4} \partial^3 W =0.$$ 
Therefore in $F^1(\cW) / F^2(\cW)$, we have the relation 
$$3 (\partial L ) W - 2 L \partial W = 0.$$ Since $\Phi_{\cW}(r_1) = 3 (\partial L ) W - 2 L \partial W$, $r_1 \in \text{ker}(\Phi_{\cW})$. 

Similarly, in $\cW$ we have the following relation in weight $8$:
$$ :(\partial W)^2 :  - \frac{9}{4} :(\partial L)^2 L:   - \frac{3}{16} :(\partial^4 L) L:   - \frac{3}{8} :(\partial^3 L)(\partial L): - \frac{9}{32} :(\partial^2 L)^2:  + \frac{1}{160} \partial^6 L=0,$$ so in $F^2(\cW) / F^3(\cW)$ we have the relation $(\partial W)^2   - \frac{9}{4} (\partial L)^2 L = 0$, and $r_2 \in \text{ker}(\Phi_{\cW})$.
\end{proof}

\begin{thm} \label{thm:isovarieties} Even though $(\tilde{X}_{\cW})_{\infty}$ and $\text{SS}(\cW)$ differ as schemes, the map of varieties
$$\text{SS}(\cW)_{\text{red}} \ra ((\tilde{X}_{\cW})_{\infty})_{\text{red}}$$ induced by $\Phi_{\cW}$, is an isomorphism.
\end{thm}

\begin{proof} It suffices to show that the map $\phi: \cO((\tilde{X}_{\cW})_{\infty}) \ra \cO(\mathbb{C}_{\infty})$ given by \eqref{eq:induced} factors through the map $\Phi_{\cW}: \cO((\tilde{X}_{\cW})_{\infty}) \ra \text{gr}(\cW)$, since $\text{ker}(\phi) = \cN$. First, the embedding $i:\cW \ra \cS$ given by \eqref{eq:embeddingW} induces a map $$\text{gr}(i): \text{gr}(\cW) \ra \text{gr}(\cS).$$ Identifying $\text{gr}(\cS)$ with $\mathbb{C}[\beta, \partial \beta, \dots, \gamma, \partial \gamma ,\dots]$, this map is given on generators by
$$\text{gr}(i) (L) =  (\beta\gamma)^2,\qquad \text{gr}(i) (W) = (\beta\gamma)^3.$$
We also have an injective map of differential graded algebras $$\psi: \cO(\mathbb{C}_{\infty}) \ra \text{gr}(\cS),$$ defined on the generator $t^{(0)}$ of  $\cO(\mathbb{C}_{\infty})$ by $\psi(t^{(0)}) = \beta\gamma$. Since $$\Phi_{\cW}(\ell^{(0)})  =  (\beta\gamma)^2 =\text{gr}(i)(L)  = \psi((t^{(0)})^2),\qquad \Phi_{\cW}(w^{(0)}) =  (\beta\gamma)^3 = \text{gr}(i)(W)  = \psi((t^{(0)})^3),$$ and $L,W$ generate $\text{gr}(\cW)$ as a differential algebra, it is clear that $$\text{gr}(i)(\text{gr}(\cW)) = \psi(A) \cong A = \phi(\cO((\tilde{X}_{\cW})_{\infty})),$$ where $A\subseteq \cO(\mathbb{C}_{\infty})$ is the subalgebra generated by $(t^{(0)})^2, (t^{(0)})^3$, and their derivatives. This completes the proof.
\end{proof}

In this example, we expect that $\text{gr}(i): \text{gr}(\cW) \ra \text{gr}(\cS)$ is injective, so that $\text{gr}(\cW)\cong A$, and in particular is reduced. However, we caution the reader that the associated graded functor is not left exact in general.

\section{Failure of associated graded functor to be left exact}
Here we give an example of a simple vertex algebra $\cV$ which has a free field realization $i: \cV \ra \cH$ where $\cH$ is the Heisenberg algebra, such that the induced map $\text{gr}(i): \text{gr}(\cV) \ra \text{gr}(\cH)$ is not injective. 

First,  $\cH$ is generated by an even field $\alpha$ satisfying $$\alpha(z) \alpha(w) \sim (z-w)^{-2},$$ and has Virasoro element $L = \frac{1}{2} :\alpha\alpha:$ of central charge $c = 1$. There is an action of $\mathbb{Z}_2$ sending $\alpha \mapsto -\alpha$ which preserves $L$, and we consider the orbifold $$\cV = \cH^{\mathbb{Z}_2}.$$  By a result of Dong and Nagatomo \cite{DN}, $\cV$ is strongly generated by $L$ together with a unique up to scalar weight $4$ field primary field 
$$W^4 = -\frac{1}{6\sqrt{6}} :\alpha^4: -\frac{1}{4 \sqrt{6}} :(\partial \alpha)^2: +\frac{1}{6\sqrt{6}} :(\partial^2 \alpha) \alpha:,$$ 
which is normalized so that it satisfies $$W^4(z) W^4(w) \sim \frac{1}{4} (z-w)^{-8} + \cdots. $$ One can check by direct calculation that $\cV$ is isomorphic to the simple, principal $\cW$-algebra of $\gs\gp_4$ with central charge $c=1$. It is convenient to replace $W^4$ with the field
$$W = \frac{35}{132} :(\partial^2 \alpha) \alpha: \  = \frac{35 \sqrt{2/3}}{33} W^4 +\frac{70}{297} :L^2:  +\frac{35}{396} \partial^2 L,$$ which is not primary. A calculation shows that we have the following nontrivial relations in $\cV$ at weights $8$ and $10$, respectively.

\begin{equation} \label{eq:hrel1} \begin{split} :W^2: - :L^2W: 
+ \frac{35}{132} :(\partial^2 L) L^2:   -  \frac{35}{264} :(\partial L)^2 L: + \frac{13265}{69696} :(\partial^4 L) L:  \\ +  \frac{19495}{139392}:( \partial^3 L) \partial L:   - \frac{59}{88} :(\partial^2 L) W:  - \frac{497}{352} :(\partial L) \partial W:   - \frac{181}{528} :L \partial^2 W:  \\ - \frac{139}{2112} \partial^4 W + \frac{10955}{557568} \partial^6 L \ =0,\end{split} \end{equation}

\begin{equation} \label{eq:hrel2} \begin{split} : L^3 W: + \frac{4455}{1024} :(\partial W) \partial W:  - \frac{35}{132} :(\partial^2 L) L^3:  + 
 \frac{35}{264} :(\partial L)^2 L^2:  + \frac{347}{256} :(\partial L)^2 W:  - 
\\  \frac{1069}{256} :(\partial L) L \partial W:  - \frac{49}{16} :L^2 \partial^2 W:  + 
 \frac{385}{576} :(\partial^4 L) L^2:  + \frac{48965}{101376} :(\partial^3 L) (\partial L) L:  - 
 \\  \frac{35}{44} :(\partial^2 L)^2 L:  + \frac{35}{88} :(\partial^2 L)(\partial L)^2:  - 
 \frac{1687}{1536} :(\partial^4 L) W:  - \frac{5939}{3072} :(\partial^3 L)(\partial W):  - 
\\  \frac{247}{256} :(\partial^2 L)(\partial^2 W):  - \frac{10927}{6144} :(\partial L) (\partial^3 W): + 
 \frac{779}{1536} :L \partial^4 W: + \frac{3899}{36864} :(\partial^6 L) L:  + 
\\  \frac{102851}{270336} :(\partial^5 L) \partial L:  + \frac{7525}{67584} :(\partial^4 L) \partial^2 L: + 
 \frac{659645}{4866048} :(\partial^3 L)^2: \\ + \frac{68311}{6488064} \partial^8 L - 
 \frac{3187}{49152} \partial^6 W =0. \end{split} \end{equation}

\begin{lemma} Let $\ell,w$ denote the images of $L,W$ in $R_{\cV}$. Then $$R_{\cV} \cong \mathbb{C}[\ell,w ] / I$$ where $I$ is the ideal generated by $w(w-\ell^2)$ and $\ell^3 w$. In particular, $\tilde{X}_{\cV} = \text{Spec}(R_{\cV})$ is irreducible of dimension one, but is not reduced.
\end{lemma}

\begin{proof} Since $\cV$ is strongly generated by $L,W$, it follows from \eqref{eq:hrel1} and \eqref{eq:hrel2} that $R_{\cV} \cong \mathbb{C}[\ell,w ] / I$ for some ideal $I$ which contains $w(w-\ell^2)$ and $\ell^3 w$. The proof that $I$ is generated by these two elements is similar to the proof of Lemma \ref{lem:rw}, and is omitted. Since $I$ is contained in the ideal $\langle w \rangle$, the map $\mathbb{C}[\ell] \ra R_{\cV}$ is injective, and $R_{\cV}$ has Krull dimension $1$. Since $w$ is a nontrivial nilpotent element of $R_{\cV}$, $\tilde{X}_{\cV}$ is not reduced. Finally, it is easy to see that the nilradical $\cN$ of $R_{\cV}$ is generated by $w$, so $\cN$ is prime and $\tilde{X}_{\cV}$ is irreducible.
\end{proof}

\begin{cor} Let $i: \cV \ra \cH$ be the inclusion. Since $\text{gr}(\cH)$ is the polynomial ring $\mathbb{C}[\alpha, \partial \alpha,\dots]$, the induced map $\text{gr}(i): \text{gr}(\cV) \ra \text{gr}(\cH)$ is not injective.  \end{cor}
In fact, it is easy to verify that the image of $\text{gr}(i)$ is just the differential polynomial algebra generated by $\text{gr}(i)(L) = \frac{1}{2} \alpha^2$. Finally, we remark that as in our main example $\cW$, the map $\Phi_{\cV}: \cO((\tilde{X}_{\cV})_{\infty})\ra  \text{gr}(\cV)$ is not injective for $\cV = \cH^{\mathbb{Z}_2}$. For example,
$$r =  \ell^{(0)}\ell^{(2)}  w^{(0)} + (\ell^{(1)})^2 w^{(0)}  - \frac{1}{2} (\ell^{(0)})^2 w^{(2)}$$ is a nontrivial element of $\text{ker}(\Phi_{\cV})$. In fact, $r$ is nilpotent in $ \cO((\tilde{X}_{\cV})_{\infty})$ and satisfies $r^3 = 0$.

\section{Universal enveloping vertex algebras}
Let $\cV$ be a conformal vertex algebra with a strong generating set $S$, i.e., for $a,b \in S$, the all terms in the OPE $a(z) b(w)$ can be expressed as normally ordered polynomials in the elements of $S$ and their derivatives. In the language of de Sole and Kac \cite{DSK}, the OPE algebra gives rise to a nonlinear conformal algebra satisfying skew-symmetry. There is a well-defined universal enveloping vertex algebra $\cU \cV$ which is the initial object in the category of vertex algebras with the above strong generating set and OPE algebra. If for all fields $a,b,c \in S$ and integers $r,s\geq 0$, the Jacobi identities
\begin{equation} \label{jacobi} a_{(r)}(b_{(s)} c) - (-1)^{|a||b|} b_{(s)} (a_{(r)}c) - \sum_{i =0}^r \binom{r}{i} (a_{(i)}b)_{(r+s - i)} c = 0,\end{equation}
hold as formal consequences of the OPE relations, this Lie conformal algebra is then called a {\it nonlinear Lie conformal algebra}. The main result (Theorem 3.9) of \cite{DSK} is that in this case, $\cU \cV$ is freely generated by $S$. This means that it has a PBW basis consisting of monomials in the elements of $S$ and their derivatives. 

In the examples $\cW$ and $\cV$ above, the universal enveloping vertex algebras are the universal $\cW_3$-algebra with $c = -2$ and the universal $\cW(\gs\gp_4, f_{\text{prin}})$-algebra with $c = 1$, respectively. Both of these are freely generated, so the associated varieties are isomorphic to $\mathbb{C}^2$ and the map \eqref{dgalgebras} is an isomorphism in both cases. It is natural to ask whether \eqref{dgalgebras} is always an isomorphism for universal enveloping vertex algebras, and in this section we provide a counterexample.

In \cite{A}, Adamovic studied a class of simple vertex algebra called $\cW(2, 2p - 1)$-algebras, where $p \geq 2$ is a positive integer.  They are strongly generated by a Virasoro field $L$ with central charge $c=1 - \frac{6(p - 1)^2}{p}$, and a weight $2p-1$ primary field $W$, and coincide with the singlet subalgebras of the $\cW_{2,p}$-triplet algebras. The triplet algebras were the first examples of $C_2$-cofinite, nonrational vertex algebras to appear in the literature \cite{AM}.

We consider the case $p = 3$, and we denote the $\cW(2,5)$-algebra by $\cA$. It can be realized explicitly inside the Heisenberg algebra $\cH$ with generator $\alpha$ as follows.
\begin{equation} \begin{split}L&= \frac{1}{2} :\alpha^2: + \sqrt{\frac{2}{3}} \partial \alpha,\\ 
W & = \frac{1}{4 \sqrt{2}} :\alpha^5:  + \frac{5}{4 \sqrt{3}}  :(\partial \alpha) \alpha^3: + \frac{5}{12 \sqrt{2}} :(\partial^2 \alpha) \alpha^2:  + \frac{5}{8 \sqrt{2}} :(\partial \alpha)^2 \alpha: \\& + \frac{5}{48 \sqrt{3}} :(\partial^3 \alpha) \alpha:  + \frac{5}{24 \sqrt{3}} :(\partial^2 \alpha)\partial \alpha:  +  \frac{1}{144 \sqrt{2}} \partial^4 \alpha.\end{split} \end{equation}
The Virasoro field $L$ has central charge $-7$, and the primary weight $5$ field $W$ satisfies
\begin{equation} \label{singletOPE} \begin{split} W(z) W(w) & \sim \frac{175}{12} (z-w)^{-10} -\frac{125}{6} L(w)(z-w)^{-8} -\frac{125}{12} \partial L(w)(z-w)^{-7} 
\\ & + \bigg(\frac{125}{3} :LL: -\frac{125}{8} \partial^2 L\bigg)(w)(z-w)^{-6} + \bigg(\frac{125}{3}  :(\partial L) L: - \frac{125}{36} \partial^3 L\bigg)(w)(z-w)^{-5}
\\ & + \bigg(50 :L^3: + \frac{25}{24} :(\partial L)^2: -25 :(\partial^2 L) L: -\frac{175}{72} \partial^4 L \bigg)(w)(z-w)^{-4} 
\\ &  + \bigg( 75 :(\partial L)L^2: -\frac{175}{8} :(\partial^2 L) \partial L:  -\frac{125}{36} :(\partial^3 L)L:  -\frac{35}{96} \partial^5 L\bigg)(w)(z-w)^{-3}
\\ & + \bigg(\frac{25}{2} :L^4:   + 
 \frac{1175}{48} :(\partial L)^2 L:
+  \frac{125}{12} :(\partial^2 L) L^2: 
 -  \frac{ 775}{128} :(\partial^2 L)^2:  
\\ & - \frac{225}{64} :(\partial^3 L) \partial L:
 - \frac{175}{64} :(\partial^4 L)L:
- \frac{1115}{13824} \partial^6 L\bigg)(z-w)^{-2}
\\ & +\bigg(25 :(\partial L) L^3:  -  \frac{25}{96} :(\partial L)^3: 
 -  \frac{125}{48} :(\partial^2 L)(\partial L) L: 
+  \frac{125}{24} :(\partial^3 L)L^2:
\\ & -  \frac{775}{288} :(\partial^3 L)\partial^2 L:
-  \frac{425}{288} :(\partial^4 L)\partial L:
-  \frac{115}{288} :(\partial^5 L)L:
-  \frac{365}{24192} \partial^7 L\bigg)(w)(z-w)^{-1}
\end{split} \end{equation}
We have the following normally ordered relations in weights $8$ and $10$, respectively.
\begin{equation} \label{singlet:rel1} 2 :L \partial W:  - 5 :(\partial L) W:  - \frac{1}{6} \partial^3 W = 0,\end{equation}
\begin{equation}\label{singlet:rel2}  \begin{split} & :W^2: - :L^5: - \frac{335}{24} :(\partial L)^2 L^2:
-\frac{25}{3} :(\partial L) L^3:  +\frac{283}{64} :(\partial^2 L) (\partial L)^2: 
\\ & +\frac{309}{64} :(\partial^2 L)^2 L:  -\frac{67}{36} :(\partial^3 L) (\partial L) L: 
+\frac{49}{216} :(\partial^3 L)^2: 
-\frac{23}{32} :(\partial^4 L) L^2: 
\\ & +\frac{49}{64} :(\partial^4 L)(\partial^2 L): 
 +\frac{249}{1280}:(\partial^5L) \partial L:  
+\frac{223}{3840} :(\partial^6 L) L: 
+\frac{1}{504} \partial^8 L = 0.\end{split} \end{equation}
It is straightforward to show using \eqref{singlet:rel1} and \eqref{singlet:rel2} that \begin{equation} R_{\cA} \cong \mathbb{C}[\ell, w] /  \langle w^2 - \ell^5\rangle.\end{equation} Here $\ell,w$ denote the images of $L,W$ in $R_{\cA}$.

Next, let $\cU = \cU \cA$ denote the universal enveloping vertex algebra of $\cA$. By abuse of notation, we shall also denote the generators of $\cU$ by $L,W$;  they satisfy the same OPE relations as the generators of $\cA$. We also denote by $\ell, w$ the images of $L,W$ in $R_{\cU}$.

\begin{lemma} $R_{\cU} \cong \mathbb{C}[\ell, w] / \langle w^2 - \ell^5\rangle \cong R_{\cA}$.
\end{lemma}
\begin{proof} Using \eqref{singletOPE}, we can compute the left side of the Jacobi identity \eqref{jacobi} in the case $a = b = c = W$, $r = 4$ and $s = 3$. We find that it does not vanish identically as a consequence of the OPE relations, but instead is given by 
\begin{equation} \label{jacobi:conseq} \frac{9075}{16} (2 :L \partial W:  - 5 :(\partial L) W:  - \frac{1}{6} \partial^3 W).\end{equation} Since all Jacobi identities must hold in any vertex algebra, \eqref{jacobi:conseq} must be a null vector, so that \eqref{singlet:rel1} holds in $\cU$. Therefore the corresponding Lie conformal algebra is {\it not} a nonlinear Lie conformal algebra, and $\cU$ is not freely generated by $L$ and $W$. Applying the operator $W_{(2)}$ to the identity \eqref{singlet:rel1} yields a nonzero multiple of the identity \eqref{singlet:rel2}. Therefore \eqref{singlet:rel2} also must hold in $\cU$, which shows that the relation $w^2 - \ell^5$ holds in $R_{\cU}$. Since $R_{\cA}$ is a quotient of $R_{\cU}$, the claim follows.
\end{proof}

\begin{remark} We expect that $\cA = \cU$, but we do not prove this.
\end{remark}

As in our previous example $\cW$, even though the scheme $X_{\cU} = \text{Spec} (R_{\cU})$ is reduced, the arc space $(X_{\cU})_{\infty}$ is not. In particular, $$r = 2 \ell^{(0)} w^{(1)} - 5  \ell^{(1)} w^{(0)}$$ is a nontrivial nilpotent element of $\cO((\tilde{X}_{\cU})_{\infty})$ satisfying $r^3 = 0$, and $r \in \text{ker} (\Phi_{\cU})$. Therefore $\cU$ is an example of a universal enveloping vertex algebra for which the map \eqref{dgalgebras} fails to be injective. 

Finally, via the embedding $$\cH \ra \cS,\qquad \alpha \mapsto \sqrt{-1} :\beta\gamma:,$$ $\cA$ can be identified with a subalgebra of $\cS$. By the same argument as Theorem \ref{thm:isovarieties}, one can check that the map on varieties induced by \eqref{dgalgebras} is an isomorphism. Therefore the same holds for $\cU$.

\end{document}